\documentclass[12pt]{amsart}
\usepackage{amsmath,amssymb,enumerate}
\usepackage{epsf,epsfig,amsfonts,graphicx,color,url}
\usepackage[autostyle]{csquotes}

\usepackage[pagewise]{lineno}%\linenumbers
\newcommand{\beq}{\begin{equation}}
\newcommand{\eeq}{\end{equation}}
\usepackage[all]{xy}
\usepackage{filecontents}

\def\R{\mathbb{R}}

\def\J{\mathcal{J}}

\def\I{\mathcal{I}}

\def\on{orthonormal  }

\numberwithin{equation}{section}

\newtheorem{theorem}{Theorem}

\newtheorem{proposition}{Proposition}

\newtheorem{lemma}{Lemma}
\newtheorem{definition}{Definition}

\renewcommand{\emph}[1]{{\bfseries\itshape{#1}}}
\numberwithin{figure}{section}

\def\sR{subRiemannian } 
\def\on{orthonormal }

\begin{document}

\newtheorem*{backgroundtheorem}{Background Theorem}

% the `*' in front gets rid of the numbering;  if I put this above \begin{document}
%formatting gets messed up

\title[No Periodic Geodesics in Jet Space]{No Periodic Geodesics in Jet Space}  
\author[A.\ Bravo-Doddoli]{Alejandro\ Bravo-Doddoli} 
\address{Alejandro Bravo-Doddoli: Dept. of Mathematics, UCSC,
1156 High Street, Santa Cruz, CA 95064}
\email{Abravodo@ucsc.edu}
\begin{abstract} 
The $J^k$ space of $k$-jets of a real function of one real variable $x$ admits the structure of a \sR manifold, which then has an associated Hamiltonian geodesic flow, and it is integrable. As in any Hamiltonian flow, a natural question is the existence of periodic solutions. Does $J^k$ have periodic geodesics? This study will find the action-angle coordinates in $T^*J^k$ for the geodesic flow and demonstrate that geodesics in $J^k$ are never periodic.  
\end{abstract}

\maketitle

\section{Introduction}

This paper is the first of three where we will prove that Carnot groups do not have periodic \sR geodesics; Enrico Le Donne made this conjecture. Here, we will establish the first case we found, which also has a simple and elegant proof. We will prove the conjecture in the case of the $k$-jets of a real function of a single variable, denoted by $J^k$. We shall publish soon the second case, namely, the meta-abelian Carnot groups, $0 = [[G,G],[G,G]]$, and the third one for a general Carnot group. All proofs have the same spirit; we will define a non-degenerate inner product in the space of polynomials on $x$ of degree bounded by $s$, where $s$ is the step of the group $G$, $J^k$ is an example of meta-abelian Carnot group with step $s = k$.

This work is the continuation of \cite{ABD,RM-ABD}, in \cite{ABD} $J^k$ was presented as \sR manifold, the \sR geodesic flow was defined, and its integrability was verified. In \cite{RM-ABD}, the \sR geodesics in $J^k$ were classified, and some of their minimizing properties were studied. The main goal of this paper is to prove:

\begin{theorem}\label{the:non-period}
$J^k$ does not have periodic geodesics.
\end{theorem}

Following the classification of geodesics from \cite{RM-ABD} (see pg. 5), the only candidates to be periodic are the ones called $x$-periodic (the other geodesics are not periodic on the $x$-coordinate); so we are focusing on the $x$-periodic geodesics.

An essential tool in \cite{RM-ABD} and this work is the bijection made by Monroy-Perez and Anzaldo-Meneses \cite{Monroy1,Monroy2,Monroy3}, also described in \cite{RM-ABD} (see pg. 4), between geodesics on $J^k$ and the pair $(F,I)$ (module translation $F(x) \to F(x-x_0)$), where $F(x)$ is a polynomial of degree bounded by $k$ and $I$ is a closed interval called Hill interval. Let us formalize its definition.
\begin{definition}
A closed interval $I$ is called Hill interval of $F(x)$, if for each $x$ inside $I$ then $F^2(x)<1$ and $F^2(x) = 1$ if $x$ is in the boundary of $I$.% The union of all the Hill interval $I$ to $F(x)$ is called the Hill region of $F(x)$. 
\end{definition}
By definition, the Hill interval $I$ of a constant polynomial $F^2(x) = c^2 < 1$ is $\R$, while, the Hill interval $I$ of the constant polynomial $F(x) = \pm 1$ is a single point. Also, $I$ is compact, if and only if, $F(x)$ is not a constant polynomial; in this case, if $I$ is in the form $[x_0,x_1]$, then $F^2(x_1)= F^2(x_0) = 1$. This terminology comes from celestial mechanics, and $I$ is the region where the dynamics governed by the fundamental equation \eqref{eq:fund} take place.

Geodesics corresponding to constant polynomials are called horizontal lines since their projection to $(x,\theta_0)$ planes are lines. In particular geodesic corresponding to $F(x) = \pm  1$ are abnormal geodesics (see \cite{tour}, \cite{BryantHsu} or \cite{monster}). Then this work will be restricted to geodesics associated with non-constant polynomials. $x$-periodic geodesics correspond to the pair  $(F,[x_0,x_1])$, where $x_0$ and $x_1$ are regular points of $F(x)$, which implies they are simple roots of $1-F^2(x)$.

\subsection*{Outline of the paper}
In Section \ref{sec:proof-th}, Proposition \ref{prop:period} is introduced and Theorem \ref{the:non-period} is proved.
The main propose of Section \ref{sec:geo-eq} is to prove Proposition \ref{prop:period}. In sub-Section \ref{sub:sec1},  the \sR structure and the \sR Hamiltonian geodesic function are introduced. In sub-Section \ref{sub:sec2}, a generating function is presented and a canonical transformation from traditional coordinates in $T^*J^k$ to action-angle coordinates $(\I,\phi)$ for the Hamiltonian systems are shown. In sub-Section \ref{sub:sec3}, Proposition \ref{prop:period} is proved.

\subsection*{Acknowledgments}

I want to express my gratitude to Enrico Le Donne for asking us about the existence of periodic geodesics and thus posing the problem.
I want to thank my advisor Richard Montgomery for his invaluable help. 
 This paper was developed with the support of the scholarship (CVU 619610) from  \enquote{Consejo de Ciencia y Tecnologia}   (CONACYT).

\section{Proof of theorem \ref{the:non-period}} \label{sec:proof-th}

Throughout the work the alternate coordinates  $(x,\theta_0,\cdots,\theta_k)$ will be used, the meaning of which meaning is introduced in the Section \ref{sec:geo-eq} and described in more detail in \cite{Monroy1,Monroy2} or \cite{RM-ABD}. $x$-periodic geodesics have the property that the change undergone by the coordinates $\theta_i$ after one $x$-period is finite and does not depend on the initial point. We summarize the above discussion with the following proposition.
\begin{proposition}\label{prop:period}
Let $\gamma(t) = (x(t),\theta_0(t),\cdots,\theta_k(t))$ in  $J^k$ be an $x$-periodic geodesic corresponding to the pair $(F,I)$. Then the $x$-period is 
\begin{equation}\label{eq:period}
L(F,I) = 2\int_{I} \frac{dx}{\sqrt{1-F^2(x)}},
\end{equation}
Moreover, it is twice the time it takes for the $x$-curve to cross its Hill interval exactly once. After one period, the changes $\Delta \theta_i := \theta_i(t_0+L) - \theta_i(t_0)$ for $i = 0,1, \dots, k$ undergone by $\theta_i$ are given by
\begin{equation}\label{eq:the-period}
\Delta \theta_i(F,I) = \frac{2}{i!} \int_{I} \frac{x^i F(x)dx}{\sqrt{1-F^2(x)}}.
\end{equation}
\end{proposition}

In \cite{RM-ABD}, a \sR manifold $\R^3_{F}$, called magnetic space, was introduced and  a similar statement like Proposition \ref{prop:period} was proved, see Proposition 4.1 from \cite{RM-ABD} (pg. 13), with an argument of classical mechanics, see \cite{Landau} page 25 equation (11.5).  

\ref{prop:period} implies that a $x$-periodic geodesic $\gamma(t)$ corresponding to the pair $(F,I)$ is periodic if and only if $\Delta \theta_i(F,I) = 0$ for all $i$.

\vskip .3cm

Because that period $L$ from equation \eqref{eq:period} is finite, we can define an inner product in the space of polynomials of degree bounded by $k$ in the following way;
\begin{equation}
 <P_1(x),P_2(x)>_{F} := \int_{I} \frac{P_1(x)P_2(x)dx}{\sqrt{1-F^2(x)}}. 
\end{equation}
This inner product is non-degenerate and will be the key to the proof of theorem \ref{the:non-period}.

\subsection{Proof of Theorem \ref{the:non-period}}

\begin{proof}
We will proceed by contradiction. Let us assume $\gamma(t)$ is a periodic geodesic on $J^k$ corresponding to the pair $(F,I)$, where $F(x)$ is not constant, then $\Delta \theta_i(F,I) = 0$ for all $i$ in $0,\cdots,k$. 

In the context of the space of polynomials of degree bounded by $k$ with inner product $< , >_{F}$, the condition $\Delta \theta_i(F,I) = 0$ is equivalent to $F(x)$ being perpendicular to $x^i$ ($0 = \Delta \theta_i(F,I) =  <x^i,F(x)>_{F}$), so  $F(x)$ being perpendicular to $x^i$ for all $i$ in $0,1,\cdots,k$. However, the set $\{x^i\}$ with $0\leq i \leq k$ is a base for the space of polynomials bounded by $k$, then $F(x)$ is perpendicular to any vector, so $F(x)$ is zero since the inner product is non-degenerate.
Being $F(x)$ equals $0$ contradicts the assumption that $F(x)$ is not a constant polynomial. 
\end{proof}

\section{Proof of proposition \ref{prop:period}}\label{sec:geo-eq}

\subsection{$J^k$ as a \sR manifold}\label{sub:sec1}
The \sR structure on $J^k$ will be here briefly described. For more details, see \cite{ABD,RM-ABD}. We see $J^k$ as $\R^{k+2}$, using $(x,\theta_0,\cdots,\theta_k)$ as global coordinates, then $J^k$ is endowed with a natural rank 2 distribution $D \subset TJ^k$ characterized by the $k$ Pafaffian equations \begin{equation}
 0 = d\theta_i - \frac{1}{i!} x^i d\theta_0, \qquad i=1,\cdots,k. 
\end{equation}
$D$ is globally framed by two vector fields 
\begin{equation}
X_1 = \frac{\partial}{\partial x}, \qquad X_2 = \sum_{i=0}^k \frac{x^i}{i!} \frac{\partial}{\partial \theta_i}.
\end{equation}
A \sR structure on $\J^k$ is defined by declaring these two vector fields to be \on.
In these coordinates the \sR metric is given by restricting $ds^2 = dx^2 + d\theta_0^2$ to $D$.

\subsubsection{Sub-Riemannian geodesic flow}
Here it is emphasized that the projections of the solution
 curves for the Hamiltonian geodesic flow are geodesics, that is, if $(p(t),\gamma(t))$ is a solution for the Hamiltonian geodesic flow then $\gamma(t)$ is a geodesic on $J^k$. 
 
Let $(p_x,p_{\theta_0},\cdots,p_{\theta_k},x,\theta_0,\cdots,\theta_k)$ be  the traditional coordinates on $T^*J^k$, or in short way as $(p,q)$. Let $P_1,P_2:T^*J^k \to \R$ be the momentum functions of the vector fields $X_1,X_2$, see \cite{tour} 8 pg or see \cite{agrachev}, in terms of the coordinates $(p,q)$ are given by 
\begin{equation}\label{eq:mom-def}
 P_1(p,q) := p_x, \qquad P_2(p,q) := \sum_{i=0}^k p_{\theta_i} \frac{x^i}{i!}.
\end{equation}

Then the Hamiltonian governing the geodesic on $J^k$ is
\begin{equation}\label{eq:ham}
H_{sR}(p,q) := \frac{1}{2}(P_1^2+P_2^2) = \frac{1}{2}p_x^2 + \frac{1}{2}(\sum_{i=0}^k p_{\theta_i} \frac{x^i}{i!})^2.
\end{equation}
It is noteworthy that $h = 1/2$ implies that the geodesic is parameterized by arc-length. It can be noticed that $H$ does not depend on $\theta_i$ for all $i$, then $p_\theta$'s define a $k+1$ constants of motion. 
\begin{lemma}\label{lem:geo-flow}
The \sR geodesic flow in $J^k$ is integrable, if $(p(t),\gamma(t))$ is a solution then
\begin{equation*}
\dot{\gamma}(t) = P_1(t) X_1 + P_2(t) X_2,\qquad (P_1(t),P_2(t)) = (p_x(t),F(x(t))), 
\end{equation*}
where $p_{\theta_i} = i! a_i$ and $F(x) = \sum_{i=0}^k a_i x^i$.
\end{lemma}

\begin{proof}
$H$ does not depend on $t$ and $\theta_i$ for all $i$, so $h := H_{sR} $ and $p_{\theta_i}$ are constants of motion, thus the Hamiltonian system is integrable. First equation form the Lemma \ref{lem:geo-flow} is consequence that $P_1$ and $P_2$ are linear in $p_x$ and $p_\theta$'s. We denote by $(a_0,\cdots,a_k)$ the level set $i! a_i = p_{\theta_i}$, then by definition of $P_1$ and $P_2$ given by equation \ref{eq:mom-def}. 
\end{proof}

\subsubsection{Fundamental equation}

The level set $(a_0,\cdots,a_k)$ defines a fundamental equation 
\begin{equation}\label{eq:fund}
H_F(p_x,x):= \frac{1}{2}p_x^2 + \frac{1}{2}F^2(x) = H|_{(a_0,\cdots,a_k)}(p,q) =  \frac{1}{2}.
\end{equation} 
Here $H_F(p_x,x)$ is a Hamiltonian function in the phase plane $(p_x,x)$, where the dynamic of $x(s)$ takes place in the Hill region $I = [x_0,x_1]$ and its solution $(p_x(t),x(t))$ with energy $h = 1/2$  lies in an algebraic curve or loop given by
\begin{equation}\label{eq:loop}
\alpha_{(F,I)} := \{(p_x,x): \frac{1}{2} = \frac{1}{2}p_x^2 + \frac{1}{2}F^2(x) \;\; \text{and} \:\: x_0 \leq x \leq x_1\},
\end{equation}  
and $\alpha_{(F,I)}$ is close and simple.

\begin{lemma}\label{lem:alp-smo}
$\alpha(F,I)$ is smooth if and only if $x_0$ and $x_1$ are regular points of $F(x)$, in other words, $\alpha(F,I)$ is smooth if and only if the corresponding geodesic $\gamma(t)$ is $x$-periodic.
\end{lemma}

\begin{proof}
A point $\alpha=(p_x,x)$ in $\alpha(F,I)$  is smooth if and only 
$$ 0 \neq \nabla H_F(p_x,x)|_{\alpha(F,I)} = (p_x,F(x)F'(x)),$$
then $\alpha$ is smooth for all $p_x \neq 0$, the points $\alpha(F,I)$ such that $p_x=0$ correspond to endpoints of the Hill interval $I$, since the condition $p_x=0$ implies $F^2(x) = 1$, the point $\alpha = (0,x_0)$  is smooth if $F'(x_0) \neq 0$, as well as, the point $\alpha = (0,x_1)$  is smooth if $F'(x_1) \neq 0$. Then $\alpha(F,I)$  is smooth if and only $x_0$ and $x_1$ are regular points of $F(x)$. Also, $\alpha(F,I)$  is smooth is equivalent to $H_F(p_x,x)|_{\alpha(F,I)}$ is never zero, which is equivalent to the Hamiltonian vector field is never zero on $\alpha(F,I)$. 
\end{proof}

\subsubsection{Arnold-Liouville manifold}

The Arnold-Liouville manifold $M|_{(a_i)}$ is given by
$$ M_{(a_0,\cdots,a_k)} := \{ (p,q) \in T^*J^k : \frac{1}{2} = H_F(p_x,x), \;\;  p_{\theta_i} = i! a_i  \}. $$
In the case  $\gamma(t)$ is $x$-periodic, $M_{\I}$ is diffeomorphic to $\mathbb{S}^1\times \R^{k+1}$,
where $\mathbb{S}^1$ is the simple closed and smooth curve  $\alpha(F,I)$. 

$\alpha(F,I)$ has two natural charts using $x$ as coordinates and given by solve the equation $H_{F} = 1/2$ with respect of $p_x$, namely, $(p_x,x) = (\pm \sqrt{1-F^2(x)},x)$. Having this in mind, 
\begin{lemma}\label{lem:alp-close}
Let $d\phi_t$ be the close one-form on $\alpha_{(F,I)}$ give by
\begin{equation}
d\phi_h := \frac{p_x}{\Pi(F,I)}|_{M_{(a_0,\cdots,a_k)}} dx = \frac{\sqrt{1-F^2(x)}}{\Pi(F,I)}  dx  ,
\end{equation}
where $\Pi(F,I)$ is the area enclosed by $\alpha(F,I)$. Then, 
$$\int_{\alpha_{(F,I)}} d\phi_h = 1 \qquad \frac{\partial}{ \partial h}  \Pi(F,I) = L(F,I).$$
as a consequence  exist the inverse function $h(\Pi)$.
\end{lemma}

\begin{proof}
Let $\Omega(F,I)$ be the closed region by $\alpha(F,I)$, then $d\phi_h$ can be extended to $\Omega(F,I)$ and Stokes' Theorem implies
\begin{equation}
\begin{split}
\Pi(F,I) := \int_{\alpha_{(F,I)}} p_xdx  & = \int_{\Omega(F,I)} dp_x \wedge dx, \\
                               & = 2 \int_{I}\sqrt{2h-F^2(x)}|_{h=\frac{1}{2}} dx.  
\end{split}
\end{equation}
This tell that $\int_{\alpha_{(F,I)}} d\phi_h = 1$, thus $d\phi_h$ is not exact.

$\Pi(F,I)$ is a function of $h$, so 
\begin{equation}
\frac{\partial }{\partial h} \Pi(F,I) =  \frac{\partial }{\partial h} \int_{I}   d\phi_h =  \int_{I}   \frac{2dx}{\sqrt{1-F^2(x)}}.
\end{equation}
\end{proof}
$\Pi(F,I)$ is also called an adiabatic invariant see \cite{Arnold} pg 297. We will use $\Pi$ when we use it as a variable and $\Pi(F,I)$ for the adiabatic invariant.

\subsection{Action-angle variables in $J^k$}\label{sub:sec2}

We will consider the actions $\I = (\Pi,a_0,\cdots,a_k)$ and find its angle coordinates $\phi = (\phi_h,\phi_{0},\cdots,\phi_{k})$, such the set $(\I,\phi)$ of coordinates are an action-angle coordinates in $J^k$.

\begin{lemma}\label{lem:act-ang}
Exist a canonical transformation $\Phi(p,q) = (\I,\phi)$, where $\phi_h$ is the local function define by the close form $d\phi_h$ from Lemma \ref{lem:alp-close} and
\begin{equation*}
\begin{split}
 \phi_i  &= -\int^x \frac{x^iF(\tilde{x}) d\tilde{x}}{\sqrt{1-F^2(\tilde{x})}} + i! \theta_i \qquad i = 0,\cdots,k. 
\end{split}
\end{equation*}
\end{lemma}

To construct the canonical transformation $\Phi(p,q)$, we will look for its generating function $S(\I,q)$, of the second type that satisfies the three following conditions. 
\begin{equation}\label{eq:gen-funct}
p = \frac{\partial S}{\partial q}, \;\;\; \phi = \frac{\partial S}{\partial \I}, \;\;\; H(\frac{\partial S}{\partial q},q) = h(\Pi) = \frac{1}{2},
\end{equation}
where $h(\Pi)$ is the function defined in Lema \ref{lem:alp-close}. For more detail on the definition of $S(\I,q)$, see \cite{Arnold} Section 50 or \cite{Landau}.

To find $S(\I,q)$, we will solve the \sR Hamilton-Jacobi equation associated with the \sR geodesic flow. For more details about the definition of this equation in \sR geometry and its relations with the Eikonal equation, see \cite{tour} 8 pg or \cite{RM-ABD}. 

\begin{proof}
The \sR Hamilton-Jacobi equation is given by
\begin{equation}\label{eq:H-J}
 h|_{1/2} = \frac{1}{2}(\frac{\partial S}{\partial x})^2 + \frac{1}{2}(\sum_{i=0}^k \frac{x^i}{i!} \frac{\partial S}{\partial \theta_i})^2.
\end{equation}

Take the ansatz
\begin{equation*}
S(\I,q)  := f(x) + \sum_{i=0}^k  i! a_i \theta_i,
\end{equation*}
as a solution. The equation \eqref{eq:H-J} becomes equation \eqref{eq:fund}, then the generating function is given by
\begin{equation}\label{H-J-sol}
\begin{split}
S(\I,q)  & = \int^{x}\sqrt{2h(\Pi)-F^2(\tilde{x})} d\tilde{x} + \sum_{i=0}^n i! a_i \theta_i \\
         & = \int^{q} (p_xdx+\sum_{i=0}^k p_{\theta_i}d\theta_i)|_{M_{(a_0,\cdots,a_k)} }. \\ 
\end{split}
\end{equation}
Here, $h(\Pi) = 1/2$ and  $S(\I,q):I \times \R^{k+1} \to \R$ is a local function. Another way to build the solutions is the second line from equation \eqref{H-J-sol}.

We can see that the first and third conditions are satisfied by construction, and the second condition for the case $i= 0,\cdots,k$. We need to verify for the action $\Pi(F,I)$. For that,   
we consider the global behavior of $S(\I,q)$; in other words, we will compute the change in $S(\I,q)$ after a loop  $\alpha(F,I)$. 
\begin{equation*}
\begin{split}
\Delta S(\Pi) =  2\int_{\alpha(F,I)} dS & = 2\int_{I} \sqrt{1-F^2(x)} dx = \Pi(F,I), \\ 
\end{split}
\end{equation*}
so $\frac{\partial}{\partial \Pi} \Delta S(\Pi) = 1$, which implies $\frac{\partial}{\partial \Pi} S(\I,q) = \phi_h$.
\end{proof}
Note: In \cite{RM-ABD} a projection $\pi_{F}:J^k\to \R^3_{F}$ was built and the solution to the \sR Hamilton-Jacobi equation on the magnetic space $\R^3_{F}$ was found. The solution given by equation \eqref{H-J-sol} is the pull-back by $\pi_F$ of the solution previously found it in $\R_{F}$, where $\pi_{F}$ is in fact, a \sR submersion.

\subsubsection{Horizontal derivative}
A horizontal derivative $\nabla_{hor}$ of a function $S:J^k \to \R$ is the unique horizontal vector field that satisfies; for every $q$ in $J^k$,
\begin{equation}
<\nabla_{hor}S,v>_{q} = dS(v),\;\;\; \text{for}\;\;v\; \in D_q,
\end{equation}
where $< , >_q$ is th \sR metric in $D_q$. For more detail see \cite{tour} pg 14-15 or \cite{agrachev}.
 
\begin{lemma}\label{lem:calib}
Let $\gamma(t)$ be  a geodesic parameterized by arc-length corresponding to the pair $(F,I)$ and $S_{F}$ the solution given by equation \eqref{H-J-sol}, then 
$$dS_{F}(\dot{\gamma})(t) = 1. $$
\end{lemma} 
\begin{proof}
Let us prove that $\dot{\gamma}(t) = (\nabla_{hor}S_F)_{\gamma(t)}$, which is just a consequence that $S_F$ is  solution  to the Hamilton Jacobi equation, that is, 
\begin{equation*}
\begin{split}
X_1(S_F)|_{\gamma(t)} &= \frac{\partial S}{\partial x}|_{\gamma(t)} = p_x(t), \\
\end{split}
\end{equation*}
but, Lemma \ref{lem:geo-flow} implies that $P_1(t) = p_x(t)$, so $P_1(t) = X_1(S_F)|_{\gamma(t)}$. As well,
\begin{equation*}
\begin{split}
X_2(S_F)|_{\gamma(t)} &= \sum_{i=0}^k \frac{x^i(t)}{i!} \frac{\partial S}{\partial \theta_i}|_{\gamma(t)} = \sum_{i=0}^k a_i x^i(t)  = F(x(t)),
\end{split}
\end{equation*}
also, Lemma \ref{lem:geo-flow} implies that $P_2(t) = F(x(t))$, so $P_2(t) = X_2(S_F)|_{\gamma(t)}$. As a consequence;
$$  \nabla_{hor}S|_{\gamma(t)}:= X_1(S_F)|_{\gamma(t)} X_1 + X_2(S_F)|_{\gamma(t)} X_2 = P_1(t) X_1 + P_2(t) X_2,$$
Lemma \ref{lem:geo-flow} implies $P_1(t) X_1 + P_2(t) X_2=\dot{\gamma}(t)$. Thus, 
$\nabla_{hor}S =\dot{\gamma}(t)$ and $dS_F(v)|_{q} = <\nabla_{hor}S_F,v>$ for all $D_q$. In particular, $$dS_F(\dot{\gamma}) = <\dot{\gamma}(t),\dot{\gamma}(t)> = 1,$$
since $t$ is arc-length parameter.
\end{proof}

\subsection{Proof of Proposition \ref{prop:period}}\label{sub:sec3}

\begin{proof}
It is well-known that the fundamental system system $H_{F}$ with energy $1/2$ has period $L(F,I)$ given by equation \eqref{eq:period} and the relation between $\Pi(F,I)$ and $L(F,I)$ is given by Lemma \ref{lem:alp-close}, see \cite{Arnold} pg 281.  
Let  $\gamma(t)$ be a $x$-periodic corresponding to $(F,I)$, we are interested in seeing the change suffered by the coordinates $\theta_i$ after one $L(I,F)$. For that, we consider the change in $S(\I,q)$ after $\gamma(t)$ travel form $t$ to $t+L(F,I)$, in other words, 
\begin{equation}\label{eq:int-sol}
\begin{split}
 L(F,I) = \int_{t}^{t+L(F,I)} dS(\dot{\gamma}(t))dt & = \Pi(F,I) + \sum_{i=0}^n i! a_i \Delta \theta_i(F,I). \\ 
\end{split}
\end{equation}
On the left side of the equation is a consequence of Lemma \ref{lem:calib}, and the right side is the integration term by term. The derivative of equation \eqref{eq:int-sol} with respect  to  $a_i$ to find  $-\frac{\partial}{\partial a_i} \Pi(F,I) = i!\Delta \theta_i$, which is equivalent to equation (\ref{eq:the-period}).

We differentiate $\Delta \theta_i := \theta_i(t+L) - \theta_i(t)$ respect to $t$, to see that $\Delta \theta_i(F,I)$ is independent of the initial point. The derivative is
\begin{equation*}
\frac{x^i(t+L)F(x(t+L))}{\sqrt{1-F^2(x(t+L))}} - \frac{x^i(t)F(x(t))}{\sqrt{1-F^2(x(t))}},
\end{equation*}
but $x(t+L) = x(t)$.
\end{proof}

\begin{appendix}

\end{appendix}

\nocite{*} % to test all bib entrys
\bibliographystyle{unsrt}
\bibliography{bibli} %

\end{document}